\documentclass[12pt,twoside]{amsart}

\usepackage[T2A]{fontenc}
\usepackage[utf8]{inputenc}

\usepackage[a4paper]{geometry}
\usepackage{amssymb}
\usepackage{amsmath}

\usepackage[sort&compress,numbers]{natbib}
\newcommand*{\doi}[1]{\href{https://doi.org/\detokenize{#1}}{doi: \detokenize{#1}}} 

\usepackage[dvipsnames,hyperref]{xcolor}

\usepackage[\GinDriver,unicode]{hyperref}

\hypersetup{
  setpagesize  = false,
  breaklinks   = true,
  pdfborder    = {0 0 0},
  colorlinks   = true, 
  urlcolor     = magenta, 
  linkcolor    = blue, 
  citecolor    = Green,  
  filecolor    = cyan   
}

\newcommand{\RR}{\mathbb R}

\newcommand{\trace}{\mathop{\mathrm{trace}}\nolimits}

\newcommand{\sign}{\mathop{\operatorname{sign}}\nolimits}

\newcommand{\ee}{\mathrm{e}}

\newcommand{\manifold}[1]{\mathcal{#1}}
\newcommand{\M}{\manifold{M}}
\newcommand{\D}{\manifold{D}}

\newcommand{\vect}[1]{\mathrm{#1}} 
\newcommand{\x}{\vect{x}}
\newcommand{\y}{\vect{y}}
\newcommand{\va}{\vect{a}}
\newcommand{\vb}{\vect{b}}

\newtheorem{thm}{Theorem}[section]
\newtheorem{prop}[thm]{Proposition}

\theoremstyle{remark}
\newtheorem{remark}{Remark}[section]
\newtheorem{example}{Example}[section]
\theoremstyle{definition}
\newtheorem{dfn}{Definition}[section]

\makeatletter
\@addtoreset{equation}{section}
\makeatother

\newcommand{\ds}{\displaystyle}

\begin{document}

\title[Minimal Lorentz surfaces in $\RR^4_2$ and $\RR^3_1$]
{Relation between the minimal Lorentz surfaces in $\RR^4_2$ and $\RR^3_1$}

\author{Krasimir Kanchev}
\author{Ognian Kassabov}
\author{Velichka Milousheva}

\address {Department of Mathematics and Informatics, Todor Kableshkov University of Transport,
158 Geo Milev Str., 1574, Sofia, Bulgaria}%
\email{kbkanchev@yahoo.com}%

\address{Institute of Mathematics and Informatics, Bulgarian Academy of Sciences,
Acad. G. Bonchev Str. bl. 8, 1113, Sofia, Bulgaria}
\email{okassabov@math.bas.bg}

\address{Institute of Mathematics and Informatics, Bulgarian Academy of Sciences,
Acad. G. Bonchev Str. bl. 8, 1113, Sofia, Bulgaria}
\email{vmil@math.bas.bg}


\begin{abstract}

In this paper we give Weierstrass-type representation formulas for the null curves and for the minimal Lorentz surfaces in the Minkowski 3-space $\RR^3_1$ using real-valued functions.
Applying the Weierstrass-type representations for the null curves, we find a correspondence between the null curves in $\RR^4_2$ and the pairs of null curves in $\RR^3_1$.
Based on  this correspondence, we obtain  a relation  between the minimal Lorentz surfaces in $\RR^4_2$
and the pairs of minimal Lorentz surfaces in $\RR^3_1$.

\end{abstract}

\maketitle



\section{Introduction} \label{S:Intro}

The study of minimal surfaces in different  ambient spaces is one of the main topics  in both classical and modern differential geometry 
and has been attracting the attention of many mathematicians for more than two centuries. 
In the last years, the geometry of Lorentz surfaces in pseudo-Euclidean spaces has been of wide interest, since pseudo-Riemannian geometry has many
interesting applications in Physics.

One of the instruments in the study of minimal surfaces in the Euclidean space is the generalized Gauss map \cite{H-O-1}. 
In \cite{Gal-Mart-Mil}, Gálvez et al. obtained a representation for spacelike surfaces in the Minkowski three-dimensional space $\RR^3_1$
using the Gauss map and the conformal structure given by the second fundamental form. Timelike surfaces in $\RR^3_1$
with prescribed Gauss curvature and Gauss map were studied in \cite{Al-Esp-Gal}, where a Kenmotsu-type representation for these surfaces was given.
In \cite{Mag}, M. Magid proved that the Gauss map and the mean curvature of a timelike surface 
satisfy a system of partial differential equations and  found a Weierstrass representation formula for timelike surfaces in $\RR^3_1$. 

A representation of a minimal Lorentz surface was given by M. Dussan and M. Magid in \cite{D-Mag}, where they solved the Bj\"orling problem for timelike surfaces in $\RR^4_2$. The Bj\"orling problem for timelike surfaces in the Lorentz-Minkowski spaces $\RR^3_1$ and $\RR^4_1$ was solved in \cite{Ch-D-Mag} and  \cite{D-Fil-Mag}, respectively. 
 Lorentz surfaces in $\RR^4_2$  were studied in \cite{Bay-Patty} and \cite{Patty} by use of spinors and Lorentz numbers (also known as para-complex, split-complex, double or hyperbolic numbers), and a generalized Weierstrass representation of a minimal Lorentz surface in $\RR^4_2$ was given in \cite{Patty}.
Many other researches have studied spacelike and timelike surfaces in a 3\nobreakdash-dimensional Lorentz space via Weierstrass-type formulas, see e.g. \cite{Cintra-Onnis}, \cite{Lira-Melo-Merc}, \cite{Ship-Pack}. 

On the other hand, the class of minimal surfaces in the Euclidean space $\RR^4$ and the classes of spacelike or timelike surfaces with vanishing mean curvature in the pseudo-Euclidean spaces  $\RR^4_1$ and $\RR^4_2$ are characterized by special systems of PDEs (called natural systems) for the  Gauss curvature $K$ and the normal curvature $\varkappa$. Explicit solutions to these systems of PDEs can be found by use of a special type of Weierstrass representation (in terms of canonical coordinates), which is called canonical Weierstrass-type representation. This problem is solved  for the class of  minimal spacelike surfaces in $\RR^4_1$ in \cite{Kanchev2017}. In the case of a minimal Lorentz surface in $\RR^4_2$, the problem is solved in \cite{K-M-1}.

A correspondence between the solutions of the system of natural PDEs describing the minimal surfaces in  $\RR^4$
and the pairs of solutions to the natural PDEs of the minimal surfaces in $\RR^3$ is
obtained in \cite{Kanchev2014}.
Using this correspondence one can find also a correspondence between the minimal surfaces in $\RR^4$
and the pairs of minimal surfaces in $\RR^3$.
Applying a similar approach, a relation between the 
maximal spacelike surfaces in $\RR^4_2$ and the pairs of maximal spacelike surfaces in $\RR^3_1$ is found in \cite{Kanchev2019}.

The main purpose of the present paper is to find a correspondence of this kind between the minimal Lorentz surfaces in $\RR^4_2$
and the pairs of minimal Lorentz surfaces in $\RR^3_1$.
As the case of Lorentz surfaces is much more complicated than the case of spacelike surfaces, 
we use a different method, applying the theory of the Weierstrass-type representations for the null curves in $\RR^4_2$, which was developed in \cite{Krasimir_Kanchev_2022126017}.
First, we use the main definitions and results from \cite{Krasimir_Kanchev_2022126017} and  obtain a correspondence between the null curves in $\RR^4_2$ and the pairs of null curves in $\RR^3_1$.

It is known, that  there exists a Weierstrass-type representation for the minimal Lorentz surfaces in $\RR^3_1$, 
analogous to the classical Weierstrass representation in $\RR^3$, see \cite{Kond}.
An analogous natural representation for the surfaces in $\RR^4_2$ exists.
In this paper, we apply the method developed in  \cite{K-M-1} and \cite{Krasimir_Kanchev_2022126017} which is based on the Weierstrass-type representation for the null curves.
This method has the obvious advantage that the surfaces and their invariants are 
expressed in terms of real-valued functions. 

In Sections 3 and 4, we recall the Weierstrass-type representation for the null curves and for the minimal Lorentz surfaces in $\RR^4_2$ and give  similar Weierstrass-type representation formulas for the null curves and for the minimal Lorentz surfaces in $\RR^3_1$ using real-valued functions. We also define canonical coordinates and give canonical Weierstrass-type representation formulas for minimal Lorentz surfaces in $\RR^3_1$.  Using the Weierstrass-type representation for the null curves, in Section 5, we give  a relation between the non-degenerate null curves in $\RR^4_2$ and $\RR^3_1$.
In Section 6, we use the canonical isotropic coordinates for the minimal Lorentz surfaces and obtain a relation between the set of minimal Lorentz 
surfaces  in $\RR^4_2$ and the set of ordered pairs of minimal Lorentz surfaces  in $\RR^3_1$.
Finally, we give formulas expressing the Gauss curvature $K$ and the normal curvature $\varkappa$ of a minimal Lorenz surface in $\RR^4_2$  in terms of the Gauss curvatures of the corresponding minimal Lorentz surfaces  in $\RR^3_1$.


\section{Preliminaries} \label{S:Prelim}

 We  denote by $\RR^3_1$, resp. $\RR^4_2$, the standard three-dimensional, resp. four-dimensional pseudo-Euclidean space
with the indefinite scalar product
\begin{equation*}\label{R^3_1-tl}
\va \cdot \vb=-a_1b_1+a_2b_2+a_3b_3,\,
\end{equation*}
resp.
\begin{equation*}\label{R^4_2-tl}
\va \cdot \vb=-a_1b_1+a_2b_2-a_3b_3+a_4b_4\,.
\end{equation*}
 
All considered functions, curves and surfaces in this paper are supposed to be real and smooth.

 Let $\M=(\D ,\x)$ be a Lorentz surface in $\RR^4_2$, where $\D\subset\RR^2$\,, and $\x : \D \to \RR^4_2$ is an immersion.
It is known (see e.g. \cite{Chen-1}) that for every minimal Lorentz surface in $\RR^4_2$ 
there exists a pair $(\alpha_1, \alpha_2)$ of null curves in $\RR^4_2$, such that
$\alpha'^2_1=0$, $\alpha'^2_2=0$, ${\alpha'_1 \cdot \alpha'_2 \neq 0}$\, and
\begin{equation} \label{MinSurf-NullCurves_R42}
	\x(t_1,t_2) = \frac{\alpha_1(t_1)+\alpha_2(t_2)}{2}\,,
\end{equation}
where $(t_1,t_2)$ is a pair of isotropic coordinates for $\M$.
 This implies that with respect to such parametrization the coefficients $E$, $F$, $G$ of the first fundamental form of $\M$ are expressed as follows:
\begin{equation*}\label{EFG-IsoCurvs}
E=G=0 \,; \qquad F = \frac{\alpha'_1 \cdot \alpha'_2}{4} \,.
\end{equation*}
Hence, the change of the sign of one of the isotropic coordinates  $(t_1,t_2)$ changes the sign of
$F$. Consequently, without loss of generality  we may  suppose that $F>0$ (or $F <0$) at each point of $\M$. 

On the other hand, if $\x$ is parametrized by isothermal coordinates, then $\M$ is minimal if and only if 
$\x$ is hyperbolic harmonic (see \cite{Krasimir_Kanchev_2022126017}).
Then, we can introduce a function $\y$, which is hyperbolic harmonic conjugate to $\x$.
The induced minimal Lorentz surface is called conjugate to the given surface.
In isotropic coordinates $(t_1,t_2)$, we have:
\begin{equation} \label{Conj_MinSurf-IsoCurves}
	\y(t_1,t_2) = \frac{\alpha_1(t_1)-\alpha_2(t_2)}{2}\,.
\end{equation}

The null curve $\alpha$ is called \emph{nondegenerate} if ${\alpha''}^2 \neq 0$ everywhere. 
Note that if we make the parameter change  $t=t(s),\, t' \neq 0$\,, in the null curve $\alpha(t)$   
then
${\alpha''_s}^2 = {\alpha''_t}^2 {t'}^4$.
This implies that the degeneracity of a null curve does not depend on the parametrization.
A minimal Lorentz surface $\M$ in $\RR^4_2$ is said to be of \emph{general type}, if
its corresponding null curves $\alpha_1$ and $\alpha_2$ are both nondegenerate
\cite{Kanchev2020}, \cite{Krasimir_Kanchev_2022126017}.
This definition does not depend on the local coordinates and is invariant 
under a motion in $\RR^4_2$.

\begin{dfn}\label{Def-MinSurf_kind123-IsoCurvs} \cite{Kanchev2020}, \cite{Krasimir_Kanchev_2022126017}
 Let $\M$ be a minimal Lorentz surface of general type in $\RR^4_2$ with corresponding null curves 
${\alpha_1}$, ${\alpha_2}$. We say that the surface $\M$ is of
\begin{itemize}
	\item  \emph{first type}, if both vectors ${\alpha''_1}$, ${\alpha''_2}$ 
	are spacelike;
	\item \emph{second type}, if both vectors ${\alpha''_1}$, ${\alpha''_2}$ 
	are timelike;
	\item \emph{third type}, if one of the vectors ${\alpha''_1}$, ${\alpha''_2}$ 
	is spacelike and the other one is timelike.
\end{itemize} 
\end{dfn}

Note that this definition does not depend on the choice of the local coordinates and is 
invariant under a motion in $\RR^4_2$. Moreover, for a surface ot third type
we may suppose ${\alpha''_1}^2>0$ and ${\alpha''_2}^2<0$\,
(otherwise we may renumerate the curves and the parameters).

 A non-degenerate curve $\alpha$ is said to be parametrized by a \emph{natural parameter},
if ${\alpha''}^2=\pm 1$\,.   Since this parameter  plays a role
similar to the role of an arc-length parameter for non-null curves, it is  known in the literature as \textit{pseudo arc-length parameter}  \cite{Vessiot1905,Duggal-Jin}. 
If $t$ is an arbitrary parameter of $\alpha$, then a natural parameter $s$ is given by
\begin{equation*}\label{NatParam_IsoCurv}
s = \int\sqrt[4]{\big|{\alpha''}^2 (t)\big|}\:dt\,.
\end{equation*}

\begin{dfn}\label{Can_Coord-MinLorSurf_R42-IsoCurvs}\cite{Krasimir_Kanchev_2022126017}
If $t_1$, resp. $t_2$ is a natural parameter of $\alpha_1$, resp. $\alpha_2$, then, we say that  $(t_1,t_2)$ 
are \emph{canonical coordinates} of the corresponding minimal Lorentz surface of general type $\M$ in $\RR^4_2$, 
(if $\M$ is of third type we suppose ${\alpha''_1}^2=1$ and ${\alpha''_2}^2=-1$\,).
\end{dfn}
  
 If $t$ and $s$ are two natural parameters of $\alpha$, then
it follows that $t'(s)=\pm 1$\,. Consequently,
\begin{equation}\label{Uniq-NatParam_IsoCurv}
t=\pm s + c\,,\qquad c=\textnormal{const}\,. 
\end{equation}
This shows that the canonical coordinates of a minimal Lorentz surface of general type in $\RR^4_2$
are defined up to numeration, sign and additive constants. For the surfaces of third type the numeration is fixed, since
${\alpha''_1}^2=1$ and ${\alpha''_2}^2=-1$\,.

Note that we can consider $\RR^3_1$ as a subspace of $\RR^4_2$, defined by $x_3=0$. 
So, every object in $\RR^3_1$ can be considered as an object in $\RR^4_2$.
Hence, the above remarks about the existence and uniqueness of the natural parameters of curves
and the canonical coordinates of the surfaces in $\RR^4_2$ hold also for $\RR^3_1$.
Moreover, the minimal Lorentz surfaces in $\RR^3_1$ considered as surfaces in $\RR^4_2$ are always of first type,
since  ${\alpha''}^2 \geq 0$ for the null curves in  $\RR^3_1$.

 The main invariants of a minimal Lorentz surface $\M$ in $\RR^4_2$ are the 
Gauss curvature $K$ and the  curvature of the normal  connection (normal curvature) $\varkappa$. 
The following statement holds true.

\begin{prop} \cite{Kanchev2020} \label{DegP_kind123-K_kappa-tl} 
 A minimal Lorentz surface $\M$ in $\RR^4_2$ is of general type if and only if 
$K^2-\varkappa^2\neq 0$\,. Moreover, $\M$ is of first or second type (resp. of third type)
if and only if $K^2-\varkappa^2>0$\, (resp. $K^2-\varkappa^2<0$\,). 
\end{prop}


\section{Weierstrass-type representation for null curves in $\RR^4_2$ and $\RR^3_1$}\label{sect_W-NullCurves_R42_R31}

The null curves in $\RR^4_2$ can be determined by a triple of real functions as follows.
 
\begin{thm}\label{prop-W_IsoCurv_R42} \cite{Krasimir_Kanchev_2022126017} 
Let $\alpha$ be a null curve in $\RR^4_2$ such that $\alpha'=(\xi_1,\xi_2,\xi_3,\xi_4)$ satisfies $\xi_1-\xi_2\neq 0$\,. 
Then, $\alpha'$ has the form 
\begin{equation}\label{W_alphap_R42}
\alpha' = f \big(\, g h+1 \,,\, g h-1 \,,\, h - g \,,\, h + g \,\big),
\end{equation}
where the functions $f\neq 0$, $g$, $h$ are uniquely determined by
\begin{equation*}\label{fgh_alphap_R42}
f=\ds\frac{1}{2}(\xi_1-\xi_2)\,; \qquad
g=\ds\frac{\xi_4-\xi_3}{\xi_1-\xi_2}\;; \qquad
h=\ds\frac{\xi_4+\xi_3}{\xi_1-\xi_2}\;.
\end{equation*}
Conversely, let $(f,g,h)$ be a triple of functions,  $f\neq 0$\,.
Then there exists a null curve $\alpha$ in $\RR^4_2$ such that \eqref{W_alphap_R42} holds and
 $\xi_1-\xi_2\neq 0$\,. 
\end{thm}

\begin{remark}\label{rem-W_IsoCurv_R42}
The restriction $\xi_1-\xi_2\neq 0$ in the last theorem is not essential for the geometric properties of the curve,
because every null curve in $\RR^4_2$ with $\xi_1-\xi_2 = 0$\,
can be transformed by a proper motion to a curve satisfying the condition $\xi_1-\xi_2\neq 0$\,.
 
\end{remark}

\smallskip
 Now, consider the null curves in $\RR^3_1$. 
Note that \eqref{W_alphap_R42} implies that $x_3=0$ is equivelent to $g=h$.
So, we obtain the following analogue of Theorem \ref{prop-W_IsoCurv_R42}:

\begin{thm}\label{prop-W_IsoCurv_R31}
Let $\alpha$ be a null curve in $\RR^3_1$, such that $\alpha'=(\xi_1,\xi_2,\xi_3)$ satisfies $\xi_1-\xi_2\neq 0$\,. 
Then, $\alpha'$ has the form: 
\begin{equation}\label{W_alphap_R31}
\alpha' = f \big(\, g^2 + 1 \,,\, g^2 - 1 \,,\, 2g \,\big),
\end{equation}
where $f\neq 0$ and $g$ are functions, given uniquely by
\begin{equation}\label{fg_alphap_R31}
f=\ds\frac{1}{2}(\xi_1-\xi_2)\,; \qquad
g=\ds\frac{\xi_3}{\xi_1-\xi_2} \;.
\end{equation}

 Conversely, let $(f,g)$ be a pair of  functions, $f\neq 0$\,.
Then, there exists a  null curve $\alpha$ in $\RR^3_1$, such that
$\alpha'$ has the form \eqref{W_alphap_R31} and
$\xi_1-\xi_2\neq 0$\,. 
\end{thm}

Formulas  \eqref{W_alphap_R31}  and  \eqref{W_alphap_R42} are analogous to the classical Weierstrass formula for minimal surfaces in $\RR^3$. 
So, we  call them the \emph{Weierstrass-type representation} of a null curve in  $\RR^3_1$, resp. $\RR^4_2$,
 \emph{gererated by the pair of functions $(f,g)$}, resp. \emph{the triple of functions $(f,g,h)$}.

 In what follows, we find the transformation formulas under  proper motions in  $\RR^3_1$ for the functions that participate in \eqref{W_alphap_R31}.  
 First, we consider a proper orthochronous motion.

\begin{thm}\label{W-proper_ortoh_move_R31} 
Let $\hat\alpha$ and $\alpha$ be null curves in $\RR^3_1$ 
having Weierstrass-type representations of the form \eqref{W_alphap_R31}.  
Then, the following conditions are equivalent:
\begin{enumerate}
	\item $\hat\alpha$ and $\alpha$ are related by a proper orthochronous motion in $\RR^3_1$ of the type\\
	$\hat\alpha(t)=A\alpha(t)+\vb$, where $A \in \mathbf{SO}^{+}(2,1,\RR)$ and  $\vb \in \RR^3_1$.
	\item The functions in the Weierstrass-type representation formulas for  
		$\hat\alpha$ and $\alpha$ are related by 	
\begin{equation}\label{hatfg_fg-proper_ortoh_move_R31}
\hat f = f(c g + d)^2 \,; 
\qquad
\hat g = \ds\frac{a g + b}{c g + d}\,; 
\qquad 
a,b,c,d \in \RR\,; \;\;
ad-bc = 1 \,.
\end{equation}
\end{enumerate}
\end{thm}

\begin{proof}
For any vector  $\x\in\RR^3_1$ we define a $2\times 2$ matrix $S$ by: 
\begin{equation*}\label{Spin_S-x_R31}
\x=(x_1,x_2,x_3)
 \ \leftrightarrow \ 
S=\left(
\begin{array}{cc}
   - x_3  & x_2 + x_1\\
   x_2 - x_1  &  x_3
\end{array}
\right) .
\end{equation*}
It is easy to see that the above correspondence  gives a linear isomorphism between $\RR^3_1$ and the linear space of all $2\times 2$
real matrices with vanishing trace. Moreover, this isomorphism satisfies the condition
$\det S = -\x^2$. So,  every linear operator, acting in the space of 
 $2\times 2$ real matrices, which preserves the determinant and the trace, 
determines an orthogonal operator in  $\RR^3_1$. 
 If $B\in\mathbf{SL}(2,\RR)$,
then $\det B S B^{-1} = \det S$ and $\trace B S B^{-1} = \trace S$.
For each such matrix $B$ there is an orthogonal matrix $A\in\mathbf{O}(2,1,\RR)$ corresponding to $B$. 
Consequently, there exists a group homomorphism $B \rightarrow A$:
\begin{equation}\label{Spin_B-A_R31}
\hat S = B S B^{-1} \ \rightarrow \ \hat\x = A\x \,.
\end{equation}
This homomorphism from  $\mathbf{SL}(2,\RR) $ in $\mathbf{O}(2,1,\RR)$ 
is a spinor map.  
The kernel of this map has two elements (see \cite{TdCastillo-1}): $I$ and $-I$,
$I$ being the unit matrix, and the image coincides with the connected component 
of the unit element in $\mathbf{O}(2,1,\RR)$. 
This is the group $\mathbf{SO}^{+}(2,1,\RR)$ of proper orthochronous motions in $\RR^3_1$. 
The form of the kernel and the image of \eqref{Spin_B-A_R31} imply that the map \eqref{Spin_B-A_R31}
gives rise to a group isomorphism:
\begin{equation*}
\mathbf{SL}(2,\RR) /\{I,-I\}\ \cong \ \mathbf{SO}^{+}(2,1,\RR)\,.
\end{equation*}
Applying this to the null curves in $\RR^3_1$, 
we derive the assertion.
\end{proof}

Every proper non-orthochronous motion in $\RR^3_1$ is a composition
of a proper orthochronous motion with a proper non-orthochronous motion,
obtained with changing the signs of the first two coordinates, 
i.e. with changing the signs of $f$ and $g$. Hence, the transformation formulas
for a proper non-orthochronous motion in $\RR^3_1$ are:
\begin{equation}\label{hatfg_fg-proper_nonortoh_move_R31}
\hat f = -f(c g + d)^2 \,; 
\qquad
\hat g = \ds\frac{a g + b}{c g + d}\,; 
\qquad 
a,b,c,d \in \RR\,; \;\;
ad-bc = -1 \,.
\end{equation}

Analogously, we can see that the transformation formulas for a non-proper orthochronous motion in $\RR^3_1$ are
\begin{equation}\label{hatfg_fg-unproper_ortoh_move_R31}
\hat f = f(c g + d)^2 \,; 
\qquad
\hat g = \ds\frac{a g + b}{c g + d}\,; 
\qquad 
a,b,c,d \in \RR\,; \;\;
ad-bc = -1 \,,
\end{equation}
and the transformation formulas for a non-proper non-orthochronous motion in $\RR^3_1$ are
\begin{equation}\label{hatfg_fg-unproper_nonortoh_move_R31}
\hat f = -f(c g + d)^2 \,; 
\qquad
\hat g = \ds\frac{a g + b}{c g + d}\,; 
\qquad 
a,b,c,d \in \RR\,; \;\;
ad-bc = 1 \,.
\end{equation}

Let $\alpha$ be a null curve in $\RR^4_2$ parametrized by an arbitrary parameter.

The following statement holds true.

\begin{prop}\label{IsoCurv-nondeg_fgh_R42} \cite{Krasimir_Kanchev_2022126017}
Let $\alpha$ be a null curve in $\RR^4_2$ with Weierstrass-type representation \eqref{W_alphap_R42}. Then,  $\alpha$ is non-degenerate if and only if $g'h' \neq 0$\, at each point. 
\end{prop}

Now, let $\alpha$ be parametrized by a natural parameter, i.e. ${\alpha''}^2=\pm 1$. Using \eqref{W_alphap_R42} we get ${\alpha''}^2 = 4f^2g'h'$ and hence, 
 $4f^2g'h'=\pm 1$\,, 
or equivalently $4f^2|g'h'|= 1$\,.
So, we obtain:
\begin{equation}\label{fgh_natparm_R42}
f = \ds\frac{\omega}{2\sqrt{|g'h'|}}\,; \qquad \omega=\pm 1 \,,
\end{equation}
and the sign of $\omega$ coinsides with the sign of $f$.
Using \eqref{W_alphap_R42} and \eqref{fgh_natparm_R42}, we obtain:
 
\begin{thm}\label{prop-W_natparm_IsoCurv_R42}\cite{Krasimir_Kanchev_2022126017}
Let $\alpha$ be a null curve in $\RR^4_2$, parametrized by a natural parameter, i.e. ${\alpha''}^2=\pm 1$. 
Then, $\alpha$ has the following Weierstrass-type representation:
\begin{equation}\label{W_natparm_alphap_R42}
\alpha'= \ds\frac{\omega}{2\sqrt{|g'h'|}}\left(gh+1, gh-1, h-g, h+g\right),
\end{equation}
where $g$ and $h$ are smooth real functions satisfying $g'h' \neq 0$, and $\omega=\pm 1$.
The functions $g$ and  $h$ as well as $\omega$ are determined uniquely by $\alpha$ in accordance with \eqref{fgh_alphap_R42}.

Conversely, if $(g,h)$ is a pair of smooth real functions satisfying $g'h' \neq 0$ and $\omega=\pm 1$, then there exists a null curve  $\alpha$ in $\RR^4_2$ parametrized by a natural parameter, such that $\alpha'$ is expressed by the given functions in the form
\eqref{W_natparm_alphap_R42}. 
\end{thm}

\begin{remark}\label{rem2-W_natparm_IsoCurv_R42}
If we apply   \eqref{W_natparm_alphap_R42} to the pair of functions $(g,h)$, we obtain two null curves
for $\omega=1$ and $\omega=-1$. These curves are related by a proper orthochronous motion in $\RR^4_2$.
\end{remark}

Note also that, if $\alpha$ is a null curve in $\RR^3_1$, then
${\alpha''}^2=4f^2g'^2$ implies the following statement:

\begin{prop}\label{IsoCurv-nondeg_fg_R31}
If $\alpha$ is a null curve in $\RR^3_1$ with Weierstrass-type representation \eqref{W_alphap_R31},
then $\alpha$ is non-degenerate if and only if $g' \neq 0$ at each point. 
\end{prop}

 If the parameter in \eqref{W_alphap_R31} is natural, then
$ {\alpha''}^2=4f^2g'^2 $ implies $2|f||g'|=1$\,. Hence, we find
\begin{equation}\label{fg_natparm_R31}
f = \ds\frac{\omega}{2|g'|}\,; \qquad \omega=\pm 1, \,
\end{equation}
where the sign of $\omega$ coincides with the sign of $f$.
Now, from \eqref{W_alphap_R31} and \eqref{fg_natparm_R31} we derive

\begin{thm}\label{prop-W_natparm_IsoCurv_R31}
Let $\alpha$ be a null curve in $\RR^3_1$, parametrized by a natural parameter, i.e. ${\alpha''}^2= 1$. 
Then, $\alpha$ has the following Weierstrass-type representation:
\begin{equation}\label{W_natparm_alphap_R31}
\alpha'= \ds\frac{\omega}{2 |g'|}\left(\,g^2+1\,,\, g^2-1\,,\, 2 g\, \right),
\end{equation}
where  $g' \neq 0$\, and $\omega=\pm 1$\,.
The function $g$ and the sign of  $\omega$ are uniquely determined by $\alpha$. 
 
Conversely, if $g$ is a  function,  $g' \neq 0$, $\omega=\pm 1$, then there exists a null curve $\alpha$ in $\RR^3_1$ parametrized by a natural parameter 
such  that \eqref{W_natparm_alphap_R31} holds. 
\end{thm}

\begin{remark}\label{rem2-W_natparm_IsoCurv_R31}
Using \eqref{W_natparm_alphap_R31} for a given function $g$ we obtain two null curves in $\RR^3_1$ corresponding to  $\omega=1 $ and $\omega=-1$, respectively.
These two curves are related  by a non-proper motion in $\RR^3_1$.
\end{remark}


\section{Weierstrass-type representation for minimal Lorentz surfaces in $\RR^4_2$ and $\RR^3_1$}\label{sect_W-MinLrntzSurf_R42_R31}

Let us recall the following result given in \cite{Krasimir_Kanchev_2022126017}.

\begin{thm}\label{prop-W_MinLorSurf_R42}
Let $\M$ be a  minimal Lorentz surface in  $\RR^4_2$ and  $(\alpha_1, \alpha_2)$  be its corresponding pair of null curves. 
Then,
\begin{equation}\label{W_MinLorSurf_R42}
\alpha'_i = f_i \big(g_i h_i + 1,  g_i h_i - 1,  h_i - g_i, h_i + g_i\big); \qquad i=1, 2,
\end{equation}
where $(f_i,g_i,h_i)$; $i=1, 2$ are  two triples of smooth real functions such that
\begin{equation}\label{W_cond_MinLorSurf_R42}
f_1(t_1)\neq 0; \qquad f_2(t_2)\neq 0; \qquad g_1(t_1) \neq g_2(t_2); \qquad h_1(t_1) \neq h_2(t_2).
\end{equation}

Conversely, if $(f_i,g_i,h_i)$; $i=1, 2$ are  two triples of  smooth real functions satisfying conditions \eqref{W_cond_MinLorSurf_R42}, then, there exists a minimal Lorentz surface in $\RR^4_2$ such that its corresponding null curves have a Weierstrass-type representation of the form  
 \eqref{W_MinLorSurf_R42} expressed by the given functions. 
\end{thm}

The Gauss curvature $K$ and the normal curvature $\varkappa$ are expressed as follows \cite{Krasimir_Kanchev_2022126017}:
\begin{equation}\label{K_kappa-fgh_MinLorSurf_R42}
\begin{array}{llr}
K         &=& \ds\frac{2}{f_1f_2 (g_1-g_2)(h_1-h_2)}
              \left(\ds\frac{g'_1g'_2}{(g_1-g_2)^2}+\ds\frac{h'_1h'_2}{(h_1-h_2)^2}\right);\\[3ex]
\varkappa &=& \ds\frac{2}{f_1f_2 (g_1-g_2)(h_1-h_2)}
              \left(\ds\frac{g'_1g'_2}{(g_1-g_2)^2}-\ds\frac{h'_1h'_2}{(h_1-h_2)^2}\right).
\end{array}
\end{equation}

\smallskip

We can obtain similar results for a minimal Lorentz surface in $\RR^3_1$,
by putting $g_i=h_i$ in the corresponding statements for a minimal surface in $\RR^4_2$. Namely, we have

\begin{thm}\label{prop-W_MinLorSurf_R31}
Let $\M$ be a minimal Lorentz surface in $\RR^3_1$ and $(\alpha_1, \alpha_2)$ be the corresponding pair of
null curves. 
Then, $\alpha'_1$ and $\alpha'_2$ have the following form:
\begin{equation}\label{W_MinLorSurf_R31}
\alpha'_i = f_i\big(\,g_i^2 + 1 \,,\, g_i^2 - 1 \,,\, 2 g_i \,\big); \qquad i=1;2\,,
\end{equation}
where $(f_i,g_i)$; $i=1;2$  are two pairs of  functions, such that: 
\begin{equation}\label{W_cond_MinLorSurf_R31}
f_1(t_1)\neq 0\,; \qquad f_2(t_2)\neq 0\,; \qquad g_1(t_1) \neq g_2(t_2) \,.
\end{equation}

 Conversely, let $(f_i,g_i)$; $i=1;2$ be two pairs of functions, satisfying \eqref{W_cond_MinLorSurf_R31}.
Then, there exists a minimal Lorentz surface in $\RR^3_1$, 
such that its corresponding curves have the form \eqref{W_MinLorSurf_R31}.
\end{thm}

For the curvatures of the surface we find
\begin{equation*}\label{K-fg_MinLorSurf_R31}
K       =        \ds\frac{4g'_1g'_2}{f_1f_2 (g_1-g_2)^4} \,, \qquad \varkappa=0 \,.
\end{equation*}

 Let $\M$ be a minimal Lorentz surface of general type in $\RR^4_2$,
parametrized by canonical isotropic coordinates, i.e. the corresponding
null curves are parametrized by natural parameters. 
Then, using Theorem \ref{prop-W_natparm_IsoCurv_R42}\,,
we obtain the next result:

\begin{thm}\label{prop-CanW_MinLorSurf_R42} \cite{Krasimir_Kanchev_2022126017}
Let $\M$ be a minimal Lorentz surface of general type in $\RR^4_2$ parametrized by canonical isotropic coordinates. Then, $\M$ has the following Weierstrass-type representation:   
\begin{equation}\label{CanW_MinLorSurf_R42}
\alpha'_i= \ds{\frac{\omega_i}{2 \sqrt{|g'_ih'_i|}}} \left(g_ih_i+1, g_ih_i-1, h_i-g_i, h_i+g_i\right); \qquad \omega_i=\pm 1,
\end{equation}
where $(g_i,h_i)$; $i=1, 2$ are two pairs of smooth real functions such that: 
\begin{equation}\label{CanW_cond_MinLorSurf_R42}
g'_1(t_1)h'_1(t_1)\neq 0; \qquad g'_2(t_2)h'_2(t_2)\neq 0; \qquad g_1(t_1) \neq g_2(t_2); \qquad h_1(t_1) \neq h_2(t_2),
\end{equation}
and, in addition, if $\M$  is of third type, then $g'_1h'_1>0$ and $g'_2h'_2<0$.

Conversely, if $(g_i,h_i)$; $i=1, 2$ are  two pairs of smooth real functions satisfying \eqref{CanW_cond_MinLorSurf_R42}, then there exists a minimal Lorentz surface of general type in $\RR^4_2$ parametrized by  canonical isotropic coordinates and having  Weierstrass-type representation
\eqref{CanW_MinLorSurf_R42}.  
\end{thm}

Formula \eqref{CanW_MinLorSurf_R42} is called  \emph{canonical Weierstrass-type representation} for 
a minimal Lorentz surface of general type in $\RR^4_2$ \cite{Krasimir_Kanchev_2022126017}.

Note that  using \eqref{CanW_MinLorSurf_R42} we can calculate that 
\begin{equation}\label{F-gh_MinLorSurf_R42-0}
F=-\frac{\omega_1\omega_2 (g_1-g_2)(h_1-h_2)}{8\sqrt{|g'_1h'_1g'_2h'_2|}}\,.
\end{equation}

If the canonical isotropic coordinates are chosen in such a way that  $F<0$, then
from \eqref{K_kappa-fgh_MinLorSurf_R42} and \eqref{fgh_natparm_R42} we obtain that the curvatures  $K$ and $\varkappa$ are expressed in terms of canonical coordinates as follows:
\begin{equation}\label{K_kappa-gh_MinLorSurf_R42}
\begin{array}{llr}
K         &=& \ds\frac{ 8  \sqrt{|g'_1h'_1g'_2h'_2|}}{|(g_1-g_2)(h_1-h_2)|}
              \left(\ds\frac{g'_1g'_2}{(g_1-g_2)^2}+\ds\frac{h'_1h'_2}{(h_1-h_2)^2}\right);\\[3ex]
\varkappa &=& \ds\frac{ 8  \sqrt{|g'_1h'_1g'_2h'_2|}}{|(g_1-g_2)(h_1-h_2)|}
              \left(\ds\frac{g'_1g'_2}{(g_1-g_2)^2}-\ds\frac{h'_1h'_2}{(h_1-h_2)^2}\right).
\end{array}
\end{equation}

\smallskip

Analogously, we have  the following theorem giving the canonical Weierstrass-type representation of a minimal Lorentz surface in $\RR^3_1$.

\begin{thm}\label{prop-CanW_MinLorSurf_R31}
Let $\M$ be a minimal Lorentz surface of general type in $\RR^3_1$, parametrized by canonical isotropic coordinates. 
Then,
\begin{equation}\label{CanW_MinLorSurf_R31}
\alpha'_i= \ds\frac{\omega_i}{2 |g'_i|} \left(\,g_i^2+1\,,\,g_i^2-1\,,\,  2g_i \,\right); \qquad \omega_i=\pm 1\,,
\end{equation}
where $g_i$; $i=1;2$ are two  functions, such that: 
\begin{equation}\label{CanW_cond_MinLorSurf_R31}
g'_1(t_1)\neq 0\,; \qquad g'_2(t_2)\neq 0\,; \qquad g_1(t_1) \neq g_2(t_2)\,.
\end{equation}

 Conversely, let $g_i$; $i=1;2$ be functions, satisfying  \eqref{CanW_cond_MinLorSurf_R31} and $\omega_i=\pm 1$\,.
Then, there exists a minimal Lorentz surface of general type  in $\RR^3_1$ parametrized by canonical isotropic coordinates and
having Weierstrass-type representation \eqref{CanW_MinLorSurf_R31}. 
\end{thm}

The Gauss curvature of  $\M$ is expressed in terms of canonical coordinates as follows:
\begin{equation}\label{K-gomega_MinLorSurf_R31}
K       =        \ds\frac{16 \omega_1\omega_2 |g'_1g'_2|g'_1g'_2}{ (g_1-g_2)^4 } \,.
\end{equation}


\section{Relation between the null curves in $\RR^4_2$ and $\RR^3_1$}\label{NullCurves_R42-NullCurves_R31}

 Let $\alpha$ be a non-degenerate null curve in $\RR^4_2$, parametrized by a natural parameter and 
having a Weierstrass-type representation \eqref{W_natparm_alphap_R42}.
Using the function  $g$, we obtain a non-degenerate 
null curve  $\alpha_g$ in $\RR^3_1$ parametrized by a natural parameter and determined by \eqref{W_natparm_alphap_R31}.
This curve $\alpha_g$ is determined up to a non-proper motion in $\RR^3_1$,  since 
$\omega_g=\pm 1$.
Analogously, the function $h$ generates a null curve $\alpha_h$ in $\RR^3_1$ parametrized by a natural parameter 
and determined up to a non-proper motion in $\RR^3_1$.
So, we obtain a map from the set of non-degenerate null curves in $\RR^4_2$ to  the
set of the ordered pairs of non-degenerate null curves in $\RR^3_1$:
\begin{equation}\label{alpha_R42-alpha_gh_R31}
		\alpha \rightarrow ( \alpha_g \,, \alpha_h ) \,.
\end{equation}
 
We will study the properties of this map under some transformations of the curves in $\RR^4_2$.

\begin{thm}\label{Invar_alpha_R42-alpha_gh_R31}
The map \eqref{alpha_R42-alpha_gh_R31} 
is invariant under a proper motion in $\RR^4_2$ and under a change of the natural parameter.
\end{thm}

\begin{proof}
First, let us suppose that the null curve $\hat\alpha$ is obtained from $\alpha$ by a proper motion in $\RR^4_2$.
According to Theorem 5.2 in \cite{Krasimir_Kanchev_2022126017},    
the functions $\hat g$ and $\hat h$ are obtained from $g$ and $h$, respectively,
by a linear-fractional transformation of determinant $\pm 1$\,, 
and both determinants have the same sign.
Now, using Theorem \ref{W-proper_ortoh_move_R31}  we derive
that $\hat\alpha_g$ is obtained from $\alpha_g$ by a motion (possibly improper) in $\RR^3_1$ 
and analogously,  $\hat\alpha_h$ is obtained from $\alpha_h$.

Now, we suppose that the null curve $\tilde\alpha$ is obtained by 
a change of the parameter $t=t(s) $\, ($t' \neq 0$) of the curve $\alpha$ in $\RR^4_2$. 
Then, we have  $\tilde\alpha'_s = \alpha'_t t'$.
Using \eqref{W_alphap_R42} we find
\begin{equation*}\label{f_g_h_s-IsoCurv_R42}
\tilde f(s) = f(t(s)) t'(s)\,; \qquad \tilde g(s) = g(t(s))\,; \qquad \tilde h(s) = h(t(s))\,.
\end{equation*}
In particular,  change \eqref{Uniq-NatParam_IsoCurv} transforms
the functions $g$ and $h$ in the canonical Weierstrass-type representation \eqref{W_natparm_alphap_R42} as follows:
\begin{equation*}\label{g_h_s-natparm_IsoCurv_R42}
\tilde g(s) = g(\pm s + c)\,; \qquad \tilde h(s) = h(\pm s + c)\,.
\end{equation*}
Using  \eqref{W_natparm_alphap_R31} we can see that the functions  $g$ and $h$ are transformed in the same way if we consider  the null curves
in $\RR^3_1$ parametrized by a natural parameter. This proves the statement.
\end{proof}

Let the null curve $\hat\alpha$ be obtained from $\alpha$ with a non-proper motion  in $\RR^4_2$.
Then, the corresponding functions $\hat g$ and $\hat h$ are obtained by changing the places of $g$ and $h$.
Consequently, \eqref{alpha_R42-alpha_gh_R31} implies:
\begin{equation*}\label{alpha_R42-alpha_gh_R31_unproper}
		\hat\alpha \rightarrow ( \alpha_h \,, \alpha_g ) \,.
\end{equation*}

\smallskip  
 Now, consider the case of a null curve $\hat\alpha$, obtained by applying an orientation-preserving anti-isometry in $\RR^4_2$ on $\alpha$.
The corresponding functions $\hat g$ and $\hat h$ are obtained from $g$ and $h$ using a linear-fractional transformations with 
determinants of different sign. Hence, \eqref{alpha_R42-alpha_gh_R31} gives:
\begin{equation}\label{alpha_R42-alpha_gh_R31_antimov}
		\hat\alpha \rightarrow ( \alpha_g \,, \alpha_{-h} ) \,.
\end{equation}

\smallskip

 Let us consider two non-degenerate null curves  $\alpha_g$ and $\alpha_h$ in $\RR^3_1$ parametrized by natural parameters
and generated  by some  functions $g$ and $h$, respectively,  in accordance with Theorem \ref{prop-W_natparm_IsoCurv_R31}.
Let $\alpha$ be the null curve in $\RR^4_2$ obtained by the Weierstrass-type representation \eqref{W_natparm_alphap_R42}\,.
Then, according to the map  \eqref{alpha_R42-alpha_gh_R31},  $\alpha$ corresponds  exactly to the initial pair $( \alpha_g \,, \alpha_h )$\,.
Consequently, the map \eqref{alpha_R42-alpha_gh_R31} is surjective.


\section{Relation between the minimal Lorentz surfaces in $\RR^4_2$ and $\RR^3_1$}\label{MinLorSurf_R42-MinLorSurfs_R31}

In the previous section, we considered a relation between the non-degenerate null curves in $\RR^4_2$ and $\RR^3_1$.
In this section, we will find a relation between the minimal Lorentz surfaces in $\RR^4_2$ 
and the pairs of minimal Lorentz surfaces in $\RR^3_1$.
We will use the canonical isotropic coordinates for the minimal Lorentz surfaces. 
In general, the canonical coordinates exist only in a neighbourhood of a point, so our considerations will be local.
To formalize the correspondence, we will consider the pairs $(\M,p)$,
where $\M$ is a minimal Lorentz surface of general type in $\RR^4_2$ and  $p\in \M$ is a fixed point.
We will identify two pairs $(\M,p)$ and $(\hat\M, \hat p)$ in $\RR^4_2$,
if  $(\hat\M,\hat p)$ is obtained from $(\M,  p)$ by a change of the coordinates and a proper motion,
$\hat p$ and $ p$ being the corresponding points.
Analogously, we identify two similar pairs in $\RR^3_1$
if one of them is obtained from the other by a change of the coordinates and a motion (possibly non-proper) in $\RR^3_1$.

 Let $\M$ be a minimal Lorentz surface of general type in $\RR^4_2$
parametrized by canonical isotropic coordinates in a neighbourhood of a point $p\in \M$.
Suppose $p$ has vanishing coordinates and $(\alpha_1, \alpha_2)$ is the pair of null curves in $\RR^4_2$ that corresponds to $\M$.
Then, $\alpha_1$ and $\alpha_2$ are non-degenerate and parametrized by natural parameters.
Applying  \eqref{alpha_R42-alpha_gh_R31} to $\alpha_1$, 
we obtain two non-degenerate null curves $\alpha_{1g}$ and $\alpha_{1h}$ in $\RR^3_1$ parametrized by natural parameters.
Analogously, $\alpha_2$ generates two non-degenerate null curves $\alpha_{2g}$ and $\alpha_{2h}$ in $\RR^3_1$, 
parametrized by natural parameters.
Then, using  \eqref{MinSurf-NullCurves_R42} we obtain two minimal Lorentz surfaces $(\M_g,p_g)$ and $(\M_h,p_h)$
of general type in $\RR^3_1$,  $p_g$ and $p_h$ being the points in $\M_g$ and $\M_h$ with vanishing coordinates.
So, we obtain a map from the set of minimal Lorentz surfaces in $\RR^4_2$ into
the set of the ordered pairs of minimal Lorentz surfaces in $\RR^3_1$:
\begin{equation}\label{MinLorSurf_R42-MinLorSurf_R31}
		\big(\M,p\big) \rightarrow \big((\M_g,p_g),(\M_h,p_h)\big) \,.
\end{equation}

 According to Theorem \ref{Invar_alpha_R42-alpha_gh_R31}, the map 
\eqref{alpha_R42-alpha_gh_R31} is invariant under a proper motion in $\RR^4_2$ and a change of the natural parameters.
The same is true for  \eqref{MinSurf-NullCurves_R42}. 
Hence, we obtain:

\begin{thm}\label{Invar_MinLorSurf_R42-MinLorSurf_R31}
The map  \eqref{MinLorSurf_R42-MinLorSurf_R31} from 
the set of minimal Lorentz surfaces of general type in $\RR^4_2$ to
the set of ordered pairs of minimal Lorentz surfaces of general type in $\RR^3_1$
is invariant under a proper motion in $\RR^4_2$ and a change of the canonical coordinates.
\end{thm}

 Note that the null curves, given with \eqref{alpha_R42-alpha_gh_R31} are defined up to a sign.
If  we change the signs of both  $\alpha_{1g}$ and $\alpha_{2g}$ in the definition of $\M_g$, then 
the new surface will be obtained from   $\M_g$ by a non-proper motion.
On the other hand, if we change only the sign of one of the curves, then according to \eqref{Conj_MinSurf-IsoCurves}
we obtain a surface   conjugate to $\M_g$. The same is true for  $\M_h$.
Consequently, the induced minimal Lorentz surfaces are defined up to a non-proper motion in $\RR^3_1$ and conjugate surface.
So, we can give the following definitions:

\begin{dfn}\label{Equiv-MinLorSurf_R31}
 Two minimal Lorentz surfaces in $\RR^3_1$ are called  \emph{equivalent},
if they coincide up to a motion (possibly non-proper) or conjugation.
\end{dfn}

\begin{dfn}\label{Equiv-MinLorSurf_R42}
Two minimal Lorentz surfaces in $\RR^4_2$ are called \emph{equivalent} 
if one of them coincides with the other (or with the conjugate of the other) up to a proper motion and an orientation-preserving anti-isometry. 
\end{dfn}

\begin{thm}\label{Equiv_MinLorSurf_R42-Equiv_MinLorSurf_R31}
Let $(\hat\M, \hat p)$ and $(\M,p)$ be equivalent minimal Lorentz surfaces in $\RR^4_2$ and $\big((\hat\M_g,\hat p_g),\ (\hat\M_h,\hat p_h)\big)$ and $\big((\M_g,p_g),\ (\M_h,p_h)\big)$ be 
the corresponding pairs of minimal Lorentz surfaces in $\RR^3_1$.
Then, $(\hat\M_g,\hat p_g)$  and $(\M_g,p_g)$  (resp. $(\hat\M_h,\hat p_h)$) and $(\M_h,p_h)$) are equivalent.
\end{thm}

\begin{proof}
Let  $(\hat\M, \hat p)$ and $(\M,p)$ be minimal Lorentz surface in $\RR^4_2$ which are equivalent.  
Theorem \ref{Invar_MinLorSurf_R42-MinLorSurf_R31} implies that if  $(\hat\M, \hat p)$ and $(\M,p)$
are related by a proper motion, then they give rise to equivalent surfaces in $\RR^3_1$.

Let $(\hat\M, \hat p)$ and $(\M,p)$ be related by an orientation-preserving anti-isometry in $\RR^4_2$\,.
According to \eqref{alpha_R42-alpha_gh_R31_antimov},  
the corresponding functions $\hat h_1$ and $\hat h_2$ are obtained by changing the sign in  $h_1$ and $h_2$. 
Consequently, \eqref{MinLorSurf_R42-MinLorSurf_R31} imp[lies
\begin{equation}\label{MinLorSurf_R42-MinLorSurf_gh_R31_antimov}
		\big(\hat\M, \hat p \big) \rightarrow \big((\M_g,p_g),(\M_{-h},p_{-h})\big) \,.
\end{equation}
Note that, according to \eqref{hatfg_fg-proper_nonortoh_move_R31} and \eqref{hatfg_fg-unproper_ortoh_move_R31}, 
if we change the signs of $h_1$ and $h_2$, then the corresponding curves are obtained by a motion (possibly non-proper) in $\RR^3_1$\,.
Hence, the surface $(\M_{-h},p_{-h})$ is obtained from $(\M_h,p_h)$ by the same motion in $\RR^3_1$\,.
Consequently, every orientation-preserving anti-isometry in $\RR^4_2$\,,
gives a pair of surfaces in $\RR^3_1$, which is equivalent to $ \big((\M_g,p_g),(\M_{h},p_{h})\big)$.

 Finally, we suppose that $(\hat\M, \hat p)$ is conjugate to $(\M,p)$.
Then, using  \eqref{Conj_MinSurf-IsoCurves}, we get that  the corresponding curves satisfy $\hat\alpha_1=\alpha_1$, $\hat\alpha_2=-\alpha_2$\,.
Using \eqref{CanW_MinLorSurf_R42},  we see that under a change of the sign of $\alpha_2$,
the functions $g_2, h_2$ do not change.
Consequently, $(\hat\M, \hat p)$ and $(\M,p)$ produce two equivalent pairs of surfaces in $\RR^3_1$\,.
\end{proof}

\begin{thm}\label{Bijektion_R42-R31}
The map \eqref{MinLorSurf_R42-MinLorSurf_R31} gives a bijection between  
the classes of equivalent minimal Lorentz surfaces in $\RR^4_2$ and
the pairs of classes of equivalent minimal Lorentz surfaces in $\RR^3_1$\,.
\end{thm}

\begin{proof}
Let us consider two minimal Lorentz surfaces of general type  $(\M_g,p_g)$ and 
$(\M_h,p_h)$ in $\RR^3_1$ parametrized by canonical coordinates.
Applying Theorem \ref{prop-CanW_MinLorSurf_R31}, we obtain two pairs of functions $(g_i,h_i)$; $i=1;2$
satisfying  \eqref{CanW_cond_MinLorSurf_R31} for $g_i$ and $h_i$\,.
Then, conditions   \eqref{CanW_cond_MinLorSurf_R42} are satisfied.
Now,  \eqref{CanW_MinLorSurf_R42} defines 
a minimal Lorentz surface of general type $(\M,p)$ in $\RR^4_2$ parametrized by canonical parameters,
$p$ being the point in $\M$ with vanishing coordinates.
It is easy to see that  \eqref{MinLorSurf_R42-MinLorSurf_R31} transforms $(\M,p)$ 
in $\big( (\M_g,p_g),(\M_h,p_h) \big)$\,. 
Consequently, the map  \eqref{MinLorSurf_R42-MinLorSurf_R31} is surjective.

 Now, we suppose that the surfaces $(\hat\M_g,\hat p_g)$ and $(\M_g,p_g)$ in $\RR^3_1$ are equivalent
and analogously, $(\hat\M_h,\hat p_h)$ and $(\M_h,p_h)$\, are equivalent.
Note, that according to \eqref{Conj_MinSurf-IsoCurves} and \eqref{CanW_MinLorSurf_R31},
the change of $(\M_g,p_g)$ or $(\M_h,p_h)$ with its conjugate
does not change the corresponding functions $g_i$ and $h_i$, $i=1;2$.

Let $(\hat\M_g,\hat p_g)$ be obtained from $(\M_g,p_g)$ by a motion (possibly non-proper) in $\RR^3_1$ and
analogously for $(\hat\M_h,\hat p_h)$ and $(\M_h,p_h)$\,.
Then, 
\eqref{hatfg_fg-proper_ortoh_move_R31},  \eqref{hatfg_fg-proper_nonortoh_move_R31},
\eqref{hatfg_fg-unproper_ortoh_move_R31}, and \eqref{hatfg_fg-unproper_nonortoh_move_R31} imply that
the functions $\hat g_i$ and $\hat h_i$ are related with $g_i$, resp. $h_i$ for $i=1;2$
by linear-fractional maps with  determinant $\pm 1$\,. 
According to Theorem 5.2 in \cite{Krasimir_Kanchev_2022126017}\,,
if the determinants for $\hat g_i$ and $\hat h_i$ have the same sign,
then the corresponding surface $(\hat\M, \hat p)$ in $\RR^4_2$ can be obtained by a proper motion  
from $(\M,p)$ or from its conjugate.
If the determinants have different signs, then having in mind  \eqref{MinLorSurf_R42-MinLorSurf_gh_R31_antimov},
$(\hat\M, \hat p)$ can be obtained from  $(\M,p)$ or its conjugate by an orientation-preserving anti-isometry in $\RR^4_2$\,.
Consequently, the corresponding surfaces $(\hat\M, \hat p)$ and $(\M,p)$ are equivalent, thus proving the injectivity.
\end{proof}

In what follows, we obtain a relation between the curvatures (Gauss curvature $K$ and
normal curvature $\varkappa$) of a minimal Lorentz surface of general type  $(\M,p)$ in $\RR^4_2$ 
and the Gauss curvatures $K_g$ and $K_h$ of the corresponding minimal Lorentz surfaces in $\RR^3_1$.

\begin{thm}\label{thm_K_kappa-KgKh_MinLorSurf_R42}
 Let $(\M,p)$ be a minimal Lorentz surface of general type in $\RR^4_2$ with 
Gauss curvature $K$ and  normal curvature  $\varkappa$.
Let   $(\M_g,p_g)$ and $(\M_h,p_h)$ be the corresponding minimal surfaces in $\RR^3_1$ with
Gauss curvatures $K_g$ and $K_h$, respectively.
\begin{enumerate}
	\item If $\M$ is of first or second type, then
	\begin{equation}\label{K_kappa_1-KgKh_MinLorSurf_R42}
K         = \eta  \sqrt[4]{|K_gK_h|} \: \frac {\sqrt{|K_g|} + \sqrt{|K_h|} }{2}\;; \quad
\varkappa = \eta  \sqrt[4]{|K_gK_h|} \: \frac {\sqrt{|K_g|} - \sqrt{|K_h|} }{2}\;, \;\; \eta=\sign K\,.
\end{equation}
	\item If $\M$ is of third type, then 
	\begin{equation}\label{K_kappa_3-KgKh_MinLorSurf_R42}
K         = \eta  \sqrt[4]{|K_gK_h|} \: \frac {\sqrt{|K_g|} - \sqrt{|K_h|} }{2}\;; \quad
\varkappa = \eta  \sqrt[4]{|K_gK_h|} \: \frac {\sqrt{|K_g|} + \sqrt{|K_h|} }{2}\;, \;\; \eta=\sign\varkappa\,.
\end{equation}
\end{enumerate}
\end{thm}

\begin{proof}
For simplicity we suppose that the first fundamental form of $\M$ satisfies $F<0$\,.
Let  $\M$ be of first or second type, i.e.  $|K|>|\varkappa|$\,.
Then, \eqref{K_kappa-gh_MinLorSurf_R42} implies that $g'_1g'_2$ and $h'_1h'_2$ have the same sign,
which coincides with the sign of $K$. Denote this sign by $\eta$.
Then, we have  $g'_1g'_2 = \eta |g'_1g'_2|$,  $h'_1h'_2 = \eta |h'_1h'_2|$.
So, equalities  \eqref{K_kappa-gh_MinLorSurf_R42} take the following form:
\begin{equation}\label{K_kappa_1-etagh_MinLorSurf_R42}
\begin{array}{llr}
K         &=& \ds\frac{ 8 \eta \sqrt{|g'_1h'_1g'_2h'_2|}}{|(g_1-g_2)(h_1-h_2)|}
              \left(\ds\frac{|g'_1g'_2|}{(g_1-g_2)^2}+\ds\frac{|h'_1h'_2|}{(h_1-h_2)^2}\right);\\[3ex]
\varkappa &=& \ds\frac{ 8 \eta \sqrt{|g'_1h'_1g'_2h'_2|}}{|(g_1-g_2)(h_1-h_2)|}
              \left(\ds\frac{|g'_1g'_2|}{(g_1-g_2)^2}-\ds\frac{|h'_1h'_2|}{(h_1-h_2)^2}\right).
\end{array}
\end{equation}

 Let  $(\M_g,p_g)$ and $(\M_h,p_h)$ be the minimal Lorentz surfaces in $\RR^3_1$ that correspond to  $(\M,p)$ 
according to \eqref{MinLorSurf_R42-MinLorSurf_R31}\,. Then, by use of 
\eqref{K-gomega_MinLorSurf_R31} we obtain that the Gauss curvatures $K_g$ and $K_h$ are expressed as follows:
\begin{equation}\label{sqrt_K-g_h_MinLorSurf_R31}
\sqrt{|K_g|}  =  \ds\frac{4 |g'_1g'_2|}{ (g_1-g_2)^2 } \,; \qquad
\sqrt{|K_h|}  =  \ds\frac{4 |h'_1h'_2|}{ (h_1-h_2)^2 } \,.
\end{equation}
Hence, using \eqref{K_kappa_1-etagh_MinLorSurf_R42}, we obtain \eqref{K_kappa_1-KgKh_MinLorSurf_R42}\,.

 Now, let $\M$ be of third type.
According to Theorem \ref{DegP_kind123-K_kappa-tl},  this is equivalent to $|\varkappa|>|K|$\,.
Then, \eqref{K_kappa-gh_MinLorSurf_R42} implies that  $g'_1g'_2$ and $h'_1h'_2$ have different signs.
If $\varkappa>0$\,, we obtain $g'_1g'_2>0$ and $h'_1h'_2<0$\,.
If $\varkappa<0$\,, we get $g'_1g'_2<0$ and $h'_1h'_2>0$\,.
Denote $\eta=\sign\varkappa$.
Then $g'_1g'_2 = \eta |g'_1g'_2|$ and $h'_1h'_2 = - \eta |h'_1h'_2|$.
Now, using  \eqref{K_kappa-gh_MinLorSurf_R42},  we derive the following expressions of $K$ and $\varkappa$:
\begin{equation}\label{K_kappa_3-etagh_MinLorSurf_R42}
\begin{array}{llr}
K         &=& \ds\frac{ 8 \eta \sqrt{|g'_1h'_1g'_2h'_2|}}{|(g_1-g_2)(h_1-h_2)|}
              \left(\ds\frac{|g'_1g'_2|}{(g_1-g_2)^2}-\ds\frac{|h'_1h'_2|}{(h_1-h_2)^2}\right);\\[3ex]
\varkappa &=& \ds\frac{ 8 \eta \sqrt{|g'_1h'_1g'_2h'_2|}}{|(g_1-g_2)(h_1-h_2)|}
              \left(\ds\frac{|g'_1g'_2|}{(g_1-g_2)^2}+\ds\frac{|h'_1h'_2|}{(h_1-h_2)^2}\right).
\end{array}
\end{equation}
From \eqref{sqrt_K-g_h_MinLorSurf_R31} and \eqref{K_kappa_3-etagh_MinLorSurf_R42} we obtain \eqref{K_kappa_3-KgKh_MinLorSurf_R42}\,.

So, we proved formulas \eqref{K_kappa_1-KgKh_MinLorSurf_R42} and \eqref{K_kappa_3-KgKh_MinLorSurf_R42} in the case $F<0$.
But, according to Theorem \ref{Invar_MinLorSurf_R42-MinLorSurf_R31}, the map \eqref{MinLorSurf_R42-MinLorSurf_R31}
do not depend on the canonical coordinates of  $\M$.
The curvatures $K$, $\varkappa$, $K_g$ and $K_h$ are also invariant. This proves the theorem.
\end{proof}

Now, we will obtain a relation between the first fundamental forms of $(\M,p)$ and
its corresponding surfaces $(\M_g,p_g)$, $(\M_h,p_h)$\,.
Using  \eqref{F-gh_MinLorSurf_R42-0}, we have 
\begin{equation}\label{F-gh_MinLorSurf_R42-1}
|F|=\frac{ |(g_1-g_2)(h_1-h_2)| }{8\sqrt{|g'_1h'_1g'_2h'_2|}}\:,
\end{equation}
 $F$ being the non-vanishing coefficient of the first fundamental form of $(\M,p)$\,.
Analogously, for the coefficients $F_g$ and $F_h$ of the surfaces  $(\M_g,p_g)$ and $(\M_h,p_h)$ we have
\begin{equation}\label{F-g_MinLorSurf_R31-1}
|F_g|=\frac{ (g_1-g_2)^2 }{8|g'_1g'_2|}\:;  \qquad  |F_h|=\frac{ (h_1-h_2)^2 }{8|h'_1h'_2|}\:.
\end{equation}
From \eqref{F-gh_MinLorSurf_R42-1} and \eqref{F-g_MinLorSurf_R31-1} we obtain:
$$
	|F| = \sqrt{|F_gF_h|}\,.
$$

\vskip 3mm
 As an application of the results presented above, at the end of this section, we give an example showing how  we can obtain a minimal Lorentz surface in $\RR^4_2$ from two given minimal Lorentz surfaces in $\RR^3_1$.

\begin{example}
 Let us consider the following Lorentz hyperbolic catenoids in $\RR^3_1$ (see \cite{Anciaux-1,Cintra-Onnis}):

\begin{itemize}
	\item $\M_a$\ :\; $x_2^2-x_1^2=\cosh^2(x_3)$ -- \emph{Lorentz hypebolic catenoid of  first kind};
	\item $\M_b$\ :\; $x_1^2-x_2^2=\sinh^2(x_3)$ -- \emph{Lorentz hypebolic catenoid of second kind}.
\end{itemize} 
 
They are parametrized respectively by:
\[
\x_a = (\sinh(u) \cosh(v), \cosh(u) \cosh(v), v); \quad \x_b = (\sinh(u) \cosh(v), \sinh(u) \sinh(v), u).
\]
It is easy to see that in both cases $\M_a$ and $\M_b$ are minimal Lorentz surfaces in $\RR^3_1$ parametrized by  isothermal coordinates $(u,v)$.
Then, we can obtain  the parametrization of $\M_a$ and $\M_b$ in terms of isotropic coordinates $(t_1,t_2)$ and in accordance with \eqref{MinSurf-NullCurves_R42} we denote the corresponding pairs of isotropic curves  by $(\alpha_1,\alpha_2)$ and $(\beta_1,\beta_2)$, respectively. 
With respect to the isotropic parameters the corresponding tangent vectors are:
\[
\begin{array}{l}
	\M_a:\;  \alpha'_1 = (\cosh(t_1), \sinh(t_1), 1);  \quad    \alpha'_2 = (\cosh(t_2), \sinh(t_2), -1);\\[0.5ex]
  \M_b:\;  \beta'_1 = (\cosh(t_1), \sinh(t_1), 1);  \quad  \!\beta'_2 = (\cosh(t_2), -\sinh(t_2), 1).
\end{array}
\]
Hence, $(\alpha''_1)^2=(\alpha''_2)^2=(\beta''_1)^2=(\beta''_2)^2=1$\, and according to Definition \ref{Can_Coord-MinLorSurf_R42-IsoCurvs}, 
$(t_1,t_2)$ are canonical isotropic coordinates of $\M_a$ and $\M_b$.

 We denote by $g_1$ and $g_2$ the functions that generate $\alpha_1$ and $\alpha_2$, respectively,  and
analogously, we denote by $h_1$ and $h_2$ the functions corresponding to  $\beta_1$ and $\beta_2$, respectively.
Applying \eqref{fg_alphap_R31} for each curve we obtain that:
\[
g_1(t_1)=\ee^{t_1}\,; \quad g_2(t_2)=-\ee^{t_2}\,; \quad h_1(t_1)=\ee^{t_1}\,; \quad h_2(t_2)=\ee^{-t_2}\,.
\]
Let us note that conditions \eqref{CanW_cond_MinLorSurf_R31} for $\M_a$ give no restriction on $(t_1,t_2)$.  
On the other hand,  conditions \eqref{CanW_cond_MinLorSurf_R31} for $\M_b$ imply $t_1+t_2 \neq 0$\,.

 Let $\gamma_1$ be the isotropic curve in $\RR^4_2$, defined by the Weierstrass-type representation \eqref{W_natparm_alphap_R42} with
the pair $(g_1,h_1)$\, and $\omega_1=1$\,.
Analogously, we denote by  $\gamma_2$ the isotropic curve determined by \eqref{W_natparm_alphap_R42} with $(g_2,h_2)$ and $\omega_2=1$\,. Then, we have:
\[
\gamma'_1 = (\cosh(t_1), \sinh(t_1),0, 1)\,;  \qquad \gamma'_2 = (0, -1, \cosh(t_2), -\sinh(t_2))\,.
\]
Denote by $\M$ the surface in $\RR^4_2$  defined by the pair $(\gamma_1,\gamma_2)$ in accordance with \eqref{MinSurf-NullCurves_R42},  or equivalently, $\M$ is the surface 
that corresponds to $(\M_a,\M_b)$ according to \eqref{MinLorSurf_R42-MinLorSurf_R31}. We have:
\[
  \M:\;  \x(t_1,t_2) = \frac{1}{2}\big(\sinh(t_1),\; \cosh(t_1)-t_2,\; \sinh(t_2),\; t_1-\cosh(t_2)\big)\,; \quad  t_1+t_2 \neq 0\,.
\]
By direct calculations we get  $(\gamma''_1)^2=(\gamma''_2)^2=1$\,. 
Hence, the surface $\M$ is of first type according to Definition \ref{Def-MinSurf_kind123-IsoCurvs}\,.

Using \eqref{fgh_natparm_R42} we find the functions $f_1$ and $f_2$  that correspond to $(\gamma'_1,\gamma'_2)$:
$$
	f_1=\frac{1}{2}\ee^{-t_1};  \qquad      f_2=\frac{1}{2} \ .
$$ 
Now, using  \eqref{K_kappa-fgh_MinLorSurf_R42} we obtain the Gauss curvature $K$ and the normal curvature $\varkappa$ of $\M$:
\begin{equation*}\label{K_kappa_Caten_R}
K         =   \frac{  -4 \cosh(t_1) \cosh(t_2)}{(\sinh(t_1) + \sinh(t_2))^3}\,; \qquad
\varkappa =   \frac{ 4-4 \cosh(t_1) \cosh(t_2)}{(\sinh(t_1) + \sinh(t_2))^3}\,.
\end{equation*}

\end{example}

\vskip 3mm
It will be interesting to study the following question: \emph{If we are given two minimal Lorentz surfaces in $\RR^3_1$,
how do the geometric properties of these two surfaces determine the
geometric properties of the corresponding surface in $\RR^4_2$?}


\vskip 5mm 
\textbf{Acknowledgments:}
The  second and the third authors are partially supported by the National Science Fund, Ministry of Education and Science of Bulgaria under contract KP-06-N52/3.

\vskip 3mm


\bibliographystyle{plainnat}

\bibliography{MyBibliography}

\begin{thebibliography}{24}
\providecommand{\natexlab}[1]{#1}
\providecommand{\url}[1]{\texttt{#1}}
\expandafter\ifx\csname urlstyle\endcsname\relax
  \providecommand{\doi}[1]{doi: #1}\else
  \providecommand{\doi}{doi: \begingroup \urlstyle{rm}\Url}\fi

\bibitem[Aledo et~al.(2006)Aledo, Espinar, and G\'alvez]{Al-Esp-Gal}
J.~A. Aledo, J.~M. Espinar, and J.~A. G\'alvez.
\newblock Timelike surfaces in the {Lorentz}–{Minkowski} space with
  prescribed {Gaussian} curvature and {Gauss} map.
\newblock \emph{Journal of Geometry and Physics}, 56\penalty0 (8):\penalty0
  1357--1369, 2006.
\newblock \doi{10.1016/j.geomphys.2005.07.004}.

\bibitem[Anciaux(2011)]{Anciaux-1}
H.~Anciaux.
\newblock \emph{Minimal Submanifolds in Pseudo-{Riemannian} Geometry}.
\newblock World Scientific Publishing Co. Pte. Ltd., 2011.
\newblock ISBN 978-981-4291-24-8.

\bibitem[Bayard and Patty(2015)]{Bay-Patty}
P.~Bayard and V.~Patty.
\newblock Spinor representation of {Lorentzian} surfaces in {${\mathbb
  R}^{2,2}$}.
\newblock \emph{J. Geom. Phys.}, 95:\penalty0 74--95, 2015.
\newblock \doi{10.1016/j.geomphys.2015.05.002}.

\bibitem[Chaves et~al.(2011)Chaves, Dussan, and Magid]{Ch-D-Mag}
R.~M.~B. Chaves, M.~P. Dussan, and M.~Magid.
\newblock {Bj\"orling} problem for timelike surfaces in the
  {Lorentz}-{Minkowski} space.
\newblock \emph{Math. Anal. Appl.}, 377\penalty0 (2):\penalty0 481--494, 2011.
\newblock \doi{10.1016/j.jmaa.2010.10.076}.

\bibitem[Chen(2011)]{Chen-1}
B.-Y. Chen.
\newblock Classification of minimal {Lorentz} surfaces in indefinite space
  forms with arbitrary codimension and arbitrary index.
\newblock \emph{Publ. Math. Debrecen}, 78\penalty0 (2):\penalty0 485--503,
  2011.
\newblock \doi{10.5486/PMD.2011.4860}.

\bibitem[Cintra and Onnis(2018)]{Cintra-Onnis}
A.~A. Cintra and I.~I. Onnis.
\newblock {Enneper} representation of minimal surfaces in the three-dimensional
  {Lorentz} - {Minkowski} space.
\newblock \emph{Annali di Matematica}, 197\penalty0 (1):\penalty0 21--39, 2018.
\newblock \doi{10.1007/s10231-017-0666-z}.

\bibitem[Duggal and Jin(2007)]{Duggal-Jin}
K.~Duggal and D.~H. Jin.
\newblock \emph{Null curves and hypersurfaces of semi-{Riemannian} manifolds}.
\newblock World Scientific, 2007.
\newblock ISBN 978-981-3106-97-0.

\bibitem[Dussan and Magid(2013)]{D-Mag}
M.~P. Dussan and M.~Magid.
\newblock The {Bj\"orling} problem for timelike surfaces in {${\mathbb
  R}^4_2$}.
\newblock \emph{J. Geom. Phys.}, 73:\penalty0 187--199, 2013.
\newblock \doi{10.1016/j.geomphys.2013.06.004}.

\bibitem[Dussan et~al.(2017)Dussan, Filho, and Magid]{D-Fil-Mag}
M.~P. Dussan, A.~P.~F. Filho, and M.~Magid.
\newblock The {Bj\"orling} problem for timelike minimal surfaces in {${\mathbb
  R}^4_1$}.
\newblock \emph{Annali di Matematica}, 196\penalty0 (4):\penalty0 1231--1249,
  2017.
\newblock \doi{10.1007/s10231-016-0614-3}.

\bibitem[G\'alvez et~al.(2003)G\'alvez, Mart{{\'\i}}nez, and
  Mil\'an]{Gal-Mart-Mil}
J.~A. G\'alvez, A.~Mart{{\'\i}}nez, and A.~Mil\'an.
\newblock Complete constant {Gaussian} curvature surfaces in the {Minkowski}
  space and harmonic diffeomorphisms onto the hyperbolic plane.
\newblock \emph{Tohoku Math. J. (2)}, 55\penalty0 (4):\penalty0 467--476, 2003.
\newblock \doi{10.2748/tmj/1113247124}.

\bibitem[Ganchev and Kanchev(2014)]{Kanchev2014}
G.~Ganchev and K.~Kanchev.
\newblock Explicit solving of the system of natural {PDE}{'}s of minimal
  surfaces in the four-dimensional {Euclidean} space.
\newblock \emph{Comptes Rendus de L{'}Academie Bulgare des Sciences},
  67\penalty0 (5):\penalty0 623--628, 2014.

\bibitem[Ganchev and Kanchev(2017)]{Kanchev2017}
G.~Ganchev and K.~Kanchev.
\newblock Explicit solving of the system of natural {PDE}{'}s of minimal
  space-like surfaces in {Minkowski} space-time.
\newblock \emph{Comptes Rendus de L{'}Academie Bulgare des Sciences},
  70\penalty0 (6):\penalty0 761--768, 2017.

\bibitem[Ganchev and Kanchev(2019)]{Kanchev2019}
G.~Ganchev and K.~Kanchev.
\newblock Relation between the maximal space-like surfaces in $\mathbb{R}^4_2$
  and the maximal space-like surfaces in $\mathbb{R}^3_1$.
\newblock \emph{Comptes Rendus de L{'}Academie Bulgare des Sciences},
  72\penalty0 (6):\penalty0 711--719, 2019.
\newblock \doi{10.7546/CRABS.2019.06.02}.

\bibitem[Ganchev and Kanchev(2020)]{Kanchev2020}
G.~Ganchev and K.~Kanchev.
\newblock Canonical coordinates and natural equations for minimal time-like
  surfaces in $\mathbb{R}^4_2$.
\newblock \emph{Kodai Mathematical Journal}, 43\penalty0 (3):\penalty0
  524--572, 2020.
\newblock \doi{10.2996/kmj/1605063628}.

\bibitem[Hoffman and Osserman(1980)]{H-O-1}
D.~Hoffman and R.~Osserman.
\newblock The geometry of the generalized {Gauss} map.
\newblock \emph{Memoirs of the American Mathematical Society}, 28\penalty0
  (236), 1980.
\newblock \doi{10.1090/memo/0236}.

\bibitem[Kanchev et~al.(2022)Kanchev, Kassabov, and
  Milousheva]{Krasimir_Kanchev_2022126017}
K.~Kanchev, O.~Kassabov, and V.~Milousheva.
\newblock Explicit solving of the system of natural {PDEs} of minimal {Lorentz}
  surfaces in $\mathbb{R}^4_2$.
\newblock \emph{Journal of Mathematical Analysis and Applications},
  510\penalty0 (1):\penalty0 126017, 2022.
\newblock ISSN 0022-247X.
\newblock \doi{10.1016/j.jmaa.2022.126017}.

\bibitem[Kassabov and Milousheva(2020)]{K-M-1}
O.~Kassabov and V.~Milousheva.
\newblock {Weierstrass} representations of {Lorentzian} minimal surfaces in
  {${\mathbb R}^4_2$}.
\newblock \emph{Mediterr. J. Math.}, 17\penalty0 (6, 199), 2020.
\newblock \doi{10.1007/s00009-020-01636-x}.

\bibitem[Konderak(2005)]{Kond}
J.~Konderak.
\newblock A {Weierstrass} representation theorem for {Lorentz} surfaces.
\newblock \emph{Complex Var. Theory Appl.}, 50\penalty0 (5):\penalty0 319--332,
  2005.
\newblock \doi{10.1080/02781070500032895}.

\bibitem[Lira et~al.(2011)Lira, Melo, and Mercuri]{Lira-Melo-Merc}
J.~H. Lira, M.~Melo, and F.~Mercuri.
\newblock A {Weierstrass} representation for minimal surfaces in 3-dimensional
  manifolds.
\newblock \emph{Results. Math.}, 60:\penalty0 311–323, 2011.
\newblock \doi{10.1007/s00025-011-0169-y}.

\bibitem[Magid(1991)]{Mag}
M.~A. Magid.
\newblock Timelike surfaces in {Lorentz} 3-space with prescribed mean curvature
  and {Gauss} map.
\newblock \emph{Hokkaido Math. J.}, 20\penalty0 (3):\penalty0 447--464, 1991.
\newblock \doi{10.14492/hokmj/1381413979}.

\bibitem[Patty(2016)]{Patty}
V.~Patty.
\newblock A generalized {Weierstrass} representation of {Lorentzian} surfaces
  in {${\mathbb R}^{2,2}$} and applications.
\newblock \emph{Methods Mod. Phys.}, 13\penalty0 (6, 1650074), 2016.
\newblock \doi{10.1142/S0219887816500742}.

\bibitem[Shipman et~al.(2017)Shipman, Shipman, and Packard]{Ship-Pack}
B.~A. Shipman, P.~D. Shipman, and D.~Packard.
\newblock Generalized {Weierstrass}–{Enneper} representations of {Euclidean},
  spacelike, and timelike surfaces: a unified {Lie}-algebraic formulation.
\newblock \emph{J. Geom.}, 108\penalty0 (2):\penalty0 545–563, 2017.
\newblock \doi{10.1007/s00022-016-0358-7}.

\bibitem[Torres~del Castillo(2010)]{TdCastillo-1}
G.~F. Torres~del Castillo.
\newblock \emph{Spinors in Four-Dimensional Spaces}.
\newblock Birkh{\"a}user Basel, 2010.
\newblock ISBN 978-0-8176-4983-8.

\bibitem[Vessiot(1905)]{Vessiot1905}
E.~Vessiot.
\newblock Sur les courbes minima.
\newblock \emph{C. R. Acad. Sci. Paris}, 140:\penalty0 1381--1384, 1905.

\end{thebibliography}

\end{document}